\documentclass[12pt]{amsart}
\usepackage{amssymb}
\usepackage{epsfig}
\usepackage{mathtools}
\usepackage{color}
\newtheorem{thm}{Theorem}[section]
\newtheorem{Prop}[thm]{Proposition}
\newtheorem{lemma}[thm]{Lemma}

\newtheorem{co}[thm]{Corollary}

\newcommand{\bc}{{\mathbb C}}

\newcommand{\bh}{{\mathbb H}}

\newcommand{\br}{{\mathbb R}}

\newcommand{\bz}{{\mathbb Z}}

\newcommand{\ra}{\rightarrow}

\newcommand{\vol}{\operatorname{vol}}

\let \cal\mathcal

\begin{document}
\title[Toledo invariant of lattice in SU(2,1)]
{Toledo invariant of lattices in SU(2,1) via symmetric square
}
\author{Inkang Kim and  Genkai Zhang}

\address{School of Mathematics,
KIAS, Heogiro 85, Dongdaemun-gu
Seoul, 130-722, Republic of Korea}
\email{inkang@kias.re.kr}
\address{Mathematical Sciences, Chalmers University of Technology and
Mathematical Sciences, G\"oteborg University, SE-412 96 G\"oteborg, Sweden}
\email{genkai@chalmers.se}

\date{}
\maketitle

\begin{abstract}In this paper, we address the issue of quaternionic Toledo invariant to study the character variety of  two dimensional complex hyperbolic uniform lattices into $SU(n,2)$. We construct four distinct representations to prove that the character variety contains at least four distinct components.
We also address the existence of
holomorphic horizontal lift to  various period domains of $SU(n,2)$.
 \end{abstract}
\footnotetext[3]{Research partially supported by
STINT-NRF grant (2011-0031291). Research by G. Zhang is supported partially
 by the Swedish
Science Council (VR). I. Kim gratefully acknowledges the partial support
of grant  (NRF-2017R1A2A2A05001002) and a warm support of
Chalmers University of Technology during his stay.
}

\section{Introduction}

After Weil's local rigidity theorem of uniform lattices in semisimple Lie groups, there have been many generalizations in different contexts.
Due to Margulis' superrigidity and Corelette's theorem, lattices in
higher rank semisimple Lie groups and in quaternionic, octonionic
hyperbolic groups are very rigid. Hence it is only meaningful to study
representations of uniform lattices $\Gamma$
in real and complex hyperbolic groups $SO(n, 1)$ respectively
$SU(n, 1)$ into different Lie groups $G$.

Several studies have been done for complex hyperbolic lattices 
$\Gamma$ in various semisimple Lie groups $G$.
In terms of maximal representations, Burger and Iozzi studied the representations of a lattice in $SU(1,p)$ with values in a Hermitian Lie group $G$ \cite{BI1, BI2}. Koziarz and Maubon \cite{KM}  studied the similar representations
in rank 2 Hermitian Lie groups. Pozzetti \cite{Po} deals with maximal representations of complex hyperbolic lattices
in $SU(m,n)$.
Recently Oscar-Garcia and Toledo \cite{GT} proved a global rigidity of complex hyperbolic lattices in quaternionic hyperbolic spaces.  More precisely, they defined the Toledo invariant $c(\rho)$ of a complex hyperbolic lattice $\Gamma$ under the representation
$\rho:\Gamma\ra PSp(n,1)$ by
$$\int_{M} f^*_\rho \omega \wedge \omega_0^{n-2}$$ where $f_\rho$ is a descended map to $M=\Gamma\backslash PSU(n,1)/S(U(n)\times U(1))$ from a  $\rho$-equivariant
map from $H^n_\bc$ to $H^n_\bh$. Here $\omega$ is the quaternionic K\"ahler form on $H^n_\bh$ and $\omega_0$ is the complex K\"ahler form
on $M$. They showed that this invariant $c(\rho)$ satisfies Milnor-Wood inequality and the maximality holds if and only if the representation
stabilizes a copy of $H^n_\bc$ inside $H^n_\bh$. Such a use of Toledo invariant goes back to Toledo \cite{Tol89} where he proves that a maximal representation from a surface group into $SU(1,q)$ fixes a complex geodesic.
Hernandez \cite{Her91} also studied maximal representations from a surface group into $SU(2,q)$ and showed that the image must stabilize a symmetric space associated to the group $SU(2,2)$. 

In this paper we attempt to generalize above results to different quaternionic K\"ahler manifolds. The first goal would be to prove a similar result in
$$\Gamma\subset SU(n,1) \subset SU(m,2)$$  using  Toledo invariant
$$c(\rho)=\int_M f_\rho^*\omega^{\frac{n}{2}}$$ for $n$ even where $\omega$ is the quaternionic K\"ahler 4-form on the associated symmetric space of $SU(m,2)$.
This Toledo invariant is constant on each connected component of the character variety
$\chi(\Gamma, SU(2n,2))$. Hence it can be used to distinguish different components of the character variety.

As a starting point, we consider the simplest case
$$\Gamma\subset SU(2,1) \subset SU(n,2),$$ $n\ge 4$.
 This case is interesting  because
the symmetric space
of   $SU(n, 2)$ has both Hermitian and quanternionic
structures and it is worth to study the interplay between them.
We will consider several different embeddings coming from the natural
  holomorphic, totally real and symmetric square representations, and obtain
\begin{thm}
There are at least 7 distinct connected components in 
$\chi(\Gamma,SU(n,2))$
 where $\Gamma\subset SU(2,1)$ is a uniform lattice.
\end{thm}
Here the group $SU(n,2)$ acts on $\text{Hom}(\Gamma, SU(n,2))$ via conjugation on the target group and the character
variety is defined by
$$\chi(\Gamma, SU(n,2))=\text{Hom}(\Gamma, SU(n,2))//SU(n,2)$$ in the sense of geometric invariant
theory.

This is one of the first examples known in higher dimensional complex
hyperbolic lattices. For different examples of character variety
$\chi(\Gamma, SU(2,1))
$, see \cite{Domingo}.
It is known in surface group case that there are $6(g-1)+1$
distinct components in $\chi(\pi_1(S), PSU(2,1))$ \cite{GKL,Xia}.
 Indeed, in \cite{GKL}, a discrete faithful representation
$\rho\in \chi(\Gamma, SU(2,1))$ is constructed such that on each component of $S\setminus \Sigma_0$, where $\Sigma_0$ is a set of disjoint simple closed geodesics, $\rho$  stabilizes either a complex line or a totally real plane. Then the Toledo invariants are maximal
on pieces contained in complex line, are zero on pieces contained in totally real plane. Hence one can realize any even integer
between $\chi(S)$ and $-\chi(S)$. This implies that there are $6(g-1)+1$ distinct components in $\chi(\pi_1(S), PSU(2,1))$.

To prove the global rigidity for $\rho\in \chi(\Gamma, G)$, the common
technique known so far is to consider a holomorphic horizontal lifting
of a $\rho$-equivariant map to a proper period domain (or twistor
space) where one can apply  complex geometric. It was successful in the case that Oscar-Garcia and Toledo considered in \cite{GT}. But in general, for higher rank case it is not known if  there always
 exists a horizontal holomorphic lifting.
\begin{thm}Consider the symmetric
square representation $\rho$ of $SU(2, 1)$ in $SU(4, 2)$ and in
$SU(n, 2)$, $n\ge 4$,
via the inclusion
 $SU(4, 2)\subseteq SU(n, 2)$.
Let $\iota:B\ra \cal X$ be the totally geodesic map induced
by the representation $\rho$
 where $B=SU(2,1)/S(U(2)\times U(1))$ and
 $\cal X=SU(n,2)/S(U(n)\times U(2))$. Then it lifts to a  holomorphic horizontal map to the period domain
$\mathcal D_2=SU(n,2)/S(U(n-1)\times U(1)\times U(2))$.
\end{thm}
See Section \ref{symmetric} for the definition of the symmetric square representation.

We thank D. Toledo for numerous discussions and suggestions for various period domains for liftability problem.
We also thank B. Klingler for a suggestion for possible different
period domains. Lastly, we thank  Mathematics department at Stanford
University where the first author spent a sabbatical year and the
second
author visited in June 2014 while part of this paper was  written.

\section{Quaternionic structure of $\mathcal X$ and its period
domains
}\label{Quaternionic}
\subsection{Quaternionic K\"ahler manifold in general}
A Riemannian manifold $M$ of real dimension $4n$ is quaternionic K\"ahler if its holonomy group is contained in $Sp(n)Sp(1)$. We denote by $\cal P_M$ the canonical $Sp(n)Sp(1)$-reduction of the principal bundle of orthogonal frames of $M$, and by $\cal E_M$ the canonical three-dimensional parallel subbundle $\cal P_M\times_{Sp(n)Sp(1)}\br^3$ of $End(TM)$. Since $Sp(n)Sp(1)$-module $\wedge^4 (\br^{4n})^*$ admits a unique trivial submodule of rank 1, any quaternionic K\"ahler manifold $M$ admits a nonzero closed 4-form $\omega$, canonical up to homothety. In \cite{Sal}, it is proved that the form $\omega$ (properly normalized) is the Chern-Weil form of the first Pontryagin class
$p_1(\cal E_M)\in H^4(M,\bz)$.

Let $N$ be a smooth closed manifold and $\rho:\pi_1(N)\ra G$ a
representation into a quaternionic K\"ahler group $G$, i.e., the
associated symmetric space $X=G/K$ is a quaternionic K\"ahler
noncompact irreducible symmetric space. Choose any $\rho$-equivariant
smooth map $\phi:\tilde N\ra  X$
from the universal covering space $\tilde N$ to $X$. The pullback
$\phi^*\cal E_X$ descends to a bundle over $N$, still denoted
$\phi^*\cal E_X$. By the functoriality of characteristic classes, the
4-form $\phi^*\omega$ represents the Pontryagin class $p_1(\phi^*\cal
E_X)\in H^4(N,\bz)$. As $X$ is contractible, any two
$\rho$-equivariant maps give rise to the same class depending only on
$\rho$. Then by the integrality of the Pontryagin class, the
quaternionic Toledo invariant $c(\rho)=\int_N
\phi^*\omega^{\frac{n}{2}}$, for even $n$,  is constant on each connected component of the character variety.

\subsection{K\"a{}hler and Quaternionic structures of $SU(2n,2)/S(U(2n)\times U(2))$}\label{negative}
Let $G=SU(p,q)$, $p\ge q$, be  in its standard realization as linear
transformations on $\mathbb C^{p+q}=
\mathbb C^{p}\oplus\mathbb C^{q}$
preserving the indefinite Hermitian form of signature $(p,q)$. 
We shall later specify $G$ to the case $SU(2n, 2)$ or $SU(n, 1)$.
Let $\cal X$ be the Hermtian symmetric space $\cal X=G/K$,
$K=S(U(p)\times U(q))$.
We recall briefly \cite{Sa} 
the Harish-Chandra realization of the symmetric
space $\cal X$ into $M_{p\times q}$  which might 
be useful in understanding various totally geodesic
embeddings in our present paper.
Fix $V^+_0=\mathbb C^{p},V^-_0=\mathbb C^q$ 
mutually orthogonal subspaces of $\bc^{p+q}$ which are positive and negative definite respectively with respect to the  Hermitian form $h^\bc$.
Fix orthonormal basis $\{e_1,\cdots,e_{p}\},\{e_{p+1}, \cdots, e_{p+q}\}$
 of $V^+_0,V^-_0$ respectively. Then $G$ acts on the set $\cal X$ of $q$-dimensional negative definite subspaces. Any other $q$-dimensional negative definite subspace $V^-$ is a graph of a unique linear map $A_{p\times q}=(z_{ij})$ from $V^-_0$ so that
$$\sum_{i=1}^{p}e_i z_{ij}+e_{p+j},\ j=1,\dots, q$$ form a basis of
$V^-$. 
Hence $\cal X$ is identified with
$$\cal X=\{Z\in M_{p\times q}: I_q - Z^t\bar Z>0\}.$$

The center of maximal compact subgroup $K$ 
is parameterized by the center of $U(p)$ and
it defines a complex K\"ahler structure. To be more precise
let $\mathfrak g$ be the Lie
algebra of  $G$, and
$\mathfrak g=\mathfrak t \oplus \mathfrak p$ its
 Cartan decomposition, where $\mathfrak k$ is the Lie algebra of $K$, with
$\mathfrak p$ consisting of matrices of the form
$$
\begin{pmatrix}
  0 & A\\
  A^* & 0 \end{pmatrix}, 
A\in M_{p\times q}.
$$ 
The real tangent space at $o=eK$ of $\cal X=G/K$
 is identified with $\mathfrak p$. The complex structure $J$ on $T_o\cal X$ acts as
  $$J \begin{pmatrix}
  0 & A\\
  A^* & 0 \end{pmatrix}=\begin{pmatrix}
       0 & iA\\
       -iA^* & 0 \end{pmatrix}.$$ The K\"ahler metric on $T_o\cal X$ is
 $$g_o(X,Y)=2\text{Tr}(YX)=4 \text{Re} \text{Tr}(B^*A),\ \text{for}\       X=\begin{pmatrix}
  0 & A\\
  A^* & 0 \end{pmatrix}, Y=\begin{pmatrix}
  0 & B\\
  B^* & 0 \end{pmatrix}.$$ 
 The corresponding complex K\"ahler form is
\begin{eqnarray}  \label{kahlerform}
\Omega_o(X,Y)=g_o(JX, Y).
\end{eqnarray}

Now let $G=SU(2n, 2)$. The second factor $U(2)$ of $K$ defines a quaternionic structure as follows.
The holomorphic tangent 
space of $\cal X$ at $o$ is identified with
$\begin{pmatrix}
0 &  Z\\
0  & 0
\end{pmatrix}$, $Z\in M_{2n\times 2}$.
The real tangent space will be parametrized and identified
with the holomorphic tangent space.
The three elements of $SU(2)$
$$\begin{pmatrix}
i &  0\\
0  & -i
\end{pmatrix} ,             \begin{pmatrix}
                            0 &  1 \\
                            -1 & 0    \end{pmatrix},   \begin{pmatrix}
                                                      0 &  i \\
                                                      i &
                                                      0\end{pmatrix}$$ 
act on the tangent space as the quaternionic multiplications
by $i,j,k$  as follows. The adjoint action of $\begin{pmatrix}
i &  0\\
0  & -i
\end{pmatrix}$ is
$$\begin{pmatrix}
I_{2n} &   0 \\
0 & \begin{pmatrix}
     -i & 0 \\
     0  & i \end{pmatrix} \end{pmatrix}      \begin{pmatrix}
                                             0 & \begin{pmatrix}
                                                   x_1 & y_1 \\
                                                   x_2 & y_2 \\
                                                   \cdots \\
                                                   x_{2n}& y_{2n} \end{pmatrix}\\
                                               0   & 0  \end{pmatrix}         \begin{pmatrix}
                                                                               I_{2n} & 0 \\
                                                                               0 & \begin{pmatrix}
                                                                                   i & 0 \\
                                                                                   0 & -i \end{pmatrix}\end{pmatrix}$$$$=\begin{pmatrix}
                                                                                       0 & \begin{pmatrix}
                                                   x_1i & y_1(-i) \\
                                                   x_2i & y_2(-i) \\
                                                   \cdots \\
                                                   x_{2n}i & y_{2n}(-i) \end{pmatrix}\\
                                               0   & 0  \end{pmatrix} . $$
 Similarly we find the action
by the other two elements. We can express the actions
in the usual quaternionic algbra, so we identify
a matrix   $(x, y)\in M_{2n\times 2}=\mathbb C^{2n}\times
\mathbb C^{2n}$ with a quaternionic vector
                                               $q\in \bh^{2n}$,
with $\mathbb H=\mathbb C +
\mathbb C j$ being the quaternionic number, by
                                               $$
X=(x, y)\leftrightarrow q_X=(x_1+ y_1 j, x_2+y_2 j, \cdots, x_{2n}+ y_{2n}j),$$ 
the previous matrix is identified with the quaternionic vector
                                  \begin{equation*}
                                              \begin{split}
                                            &\quad\,
       (x_1i+y_1(-i)j, x_2i+y_2(-i)j,\cdots,x_{2n}i+y_{2n}(-i)j)\\
&=      (x_1+y_1j, x_2+y_2j,\cdots,x_{2n}+y_{2n}j)i =q_Xi,
                                                 \end{split}
                                               \end{equation*}
i.e.,  the adjoint action of $\begin{pmatrix}
                                                   i & 0\\
                                                   0 &
                                                   -i \end{pmatrix}$ 
is just the multiplication $q_X\mapsto q_Xi$ by $i$ on the right.
It is easy to check that the adjoint action of the other two elements correspond to the multiplication by $j$ and $k$ on the right.
When no confusion would arise we shall just write the identification
$Z\to q_Z$ as $q_Z =Z$.

The  parallel closed nondegenerate quaternionic K\"ahler 4-form, at the origin is given by
\begin{eqnarray}\label{kahler}\omega= \omega_i\wedge \omega_i +
  \omega_j 
\wedge \omega_j +
\omega_k \wedge \omega_k\end{eqnarray}  where
$$\omega_u(X,Y)=\text{Re}( q_X\cdot {\bar q_Y} u), \quad u=i, j, k,
$$ and
$p\cdot \bar q=\sum_{m=1}^{2n} p_m \bar q_m$ is the standard quaternionic Hermitian form
on $\bh^{2n}$ and
$\text{Re}\ x=x_0$ is the real part
of a quaternionic number
$x=x_0+x_1i+x_2j +x_3 k$.


Then it is easy to check that this $\omega$ and $\Omega_o^2$, where
$\Omega_o$
 is the complex K\"ahler form on $\cal X$
defined above, are linearly independent on $H^4(M,\br)$ where $M=\Gamma\backslash \cal X$.

\subsection{Twister space and Period domain
of the quaternionic structures of $SU(n,2)/S(U(n)\times U(2))$
}\label{period}

We describe  one  twister space and
one period domain
for the quaternionic structures of $G/K=SU(n, 2)/S(U(n)\times U(2))$
which are not $G$-equivalent.
By $G$-equivalent we mean there exists a $G$-invariant
biholomorphic mapping between them.

For any Lie algebra $\mathfrak s$ we denote
its complexification by $\mathfrak s^{\mathbb C}$.
Let $\mathcal D_1=SU(n, 2)/S(U(n)\times U(1)\times U(1))$
be a twistor space.
We shall realize it as an open
subset in a homogeneous flag manifold.
Let $W=\mathbb C^{n+2}$ and let
$W^\ast$ be the dual space equipped with the
$G$-invariant metric of signature $(n, 2)$.
Denote
 $\{\epsilon_j\}$ in $W^\ast$
 the dual basis of $\{E_j\}$.
Let $\mathcal D_1^c$ 
be the set
of orthogonal pairs $(l, \lambda)$
in
 $\mathbb P(W)\times \mathbb  P(W^\ast)
$,
 i.e.,
satisfying  $\epsilon(e)=0$ for all $(e, \epsilon)
\in l\times \lambda$. Then 
 $\mathcal D_1^c$  is a compact homogeneous space of 
$SU(n+2)$,  $\mathcal D_1^c=SU(n+2)/S(U(n)\times U(1)\times
U(1))$. 
As a homogeneous
manifold of $SU(n,2)$, $\mathcal D_1=SU(n, 2)/S(U(n)\times U(1)\times U(1))$
can be realized as the open domain in $\mathcal D_1^c$ of
 $(l, \lambda)$ such that
$l$ and $\lambda$ are negative
definite.
Indeed, first it is elementary
to see that $SU(n, 2)$ 
acts transitively on the subset of lines.
Second we need to check that a stabilizer of $(l,\lambda)$ is
$S(U(n)\times U(1)\times U(1))$.
 A stabilizer of the negative 2-plane 
$l+(\ker \lambda)
^{\perp}$ in $W$
is $S(U(n)\times U(2))$
 and a stabilizer 
in  $S(U(n)\times U(2))$ of the
pair $(l, \ker \lambda)$
 of subspaces in $W$, equivalently
the pair  $(l,  \lambda)$ in 
 $\mathbb P(W)\times \mathbb  P(W^\ast)
$, is
exactly $U(1)\times U(1)$. Hence as a differentiable manifold
$\mathcal D_1$ has such a realization.

Then   $\mathcal D_1\subset \mathbb P(W)\times \mathbb  P(W^\ast)$
is an open subset equipped with the corresponding complex structure.

In general if a homogeneous manifold $G/(L\times U(1))$ has $U(1)$ factor in the stabilizer, it inherits a complex structure as follows.
Let $\mathfrak u(1)=\br i H_1$ and consider the root space decomposition of $\mathfrak g^\bc$ under the action of 
$$H_1=\begin{pmatrix}
0  & 0 \\
0 & \begin{pmatrix}
 1 & 0 \\
 0 & -1 \end{pmatrix} \end{pmatrix}.
$$
 Set $\mathfrak b$ to be the Borel subalgebra consisting of zero and
 negative eigenspaces. The positive eigenspace $\mathfrak n^+$ constitutes the
 holomorphic tangent space for $G^\bc/B$ 
at the base point $eB\in
G^{\bc}/B$, and further on the whole space
$G^\bc/B$, in particular for open set $$G/(L\times U(1)) \subset G^\bc/B.$$

We find the holomorphic tangent space of $
\cal D_1$ in this context.
To find a realization of the complex tangent
space we fix the pair
$(\mathbb C E_{n+2},
\mathbb C \epsilon_{n+1})$ as a base point of
$\mathcal  D_1$. The space
$\cal D_1$ is an open subset of the complex homogeneous
space of $SL(n+2, \mathbb C)/B$, 
where $B$ is
the Borel subgroup whose Lie algebra consists
of elements in $\mathfrak{sl}(n+2, \mathbb C)$
of the special form.

To justify this, note that $B$ is equal to the stabilizer of $(\mathbb C E_{n+2},
\mathbb C \epsilon_{n+1})$.
Hence $B$ should have the block matrix of form, 
$$\begin{pmatrix}
* & * & 0 \\
0 & * &  0 \\
* & * & *   \end{pmatrix},
$$
the size of the  matrix being $(n+1+1)\times (n+1+1)$.
Alternatively $\mathfrak b$ consists of  
non-positive root  spaces 
of $H_1$, i.e. eigenspaces of $ad(H_1)$.
Thus holomorphic
tangent space $\mathfrak n^+$ consists of elements of $\mathfrak g^{\mathbb C}$
of the form, the size of the matrix being 
the same as above,
$$\begin{pmatrix}
0 & 0  & * \\
* &  0 & * \\
0 & 0 &  0 \end{pmatrix}.
$$

We consider now another domain $\mathcal D_2
=SU(n, 2)/S(U(n-1)\times U(1)\times U(2))
$. More precisely
let $\{E_1, \cdots, E_n; E_{n+1}, E_{n+2}\}$ be 
the standard basis of $\mathbb C^{n+2}$ as before and
$$
H_2=\text{diag}(1, \cdots, 1, -n, 1, 0, 0)\in \mathfrak{k}^{\mathbb C}
$$
and let $U(1)=\exp(i\mathbb RH_2)$
be the corresponding subgroup of $K$.
The centralizer of $H_2$ in $K$ is then $U(n-1)\times U(1)\times
U(2)$.
Here $U(n-1)$ stands for the unitary group of the subspace
$\mathbb C^{n-1}:=\langle E_1, \cdots, E_{n-2}, E_n\rangle 
$.

Now  the  eigenspaces of positive eigenvalues
of  $\text{ad}(H_2)$ constitute
the holomorphic tangent space of $\cal D_2$:
$$\begin{pmatrix}
 0 & * & 0 & *  \\
 0 & 0 & 0 & 0  \\
 0 & * & 0 & *  \\
 0 & * & 0 & 0  \end{pmatrix}
$$
written in block form of size $((n-2) + 1 + 1 + 2)\times
((n-2) + 1 + 1 + 2)$.

The compact homogeneous space
$\mathcal D_2^c:
=SU(n+2)/S(U(n-1)\times U(1)\times U(2))$
is precisely the partial flag manifold
of pairs $(p_1, p_2)$
of subspaces  $p_1\subset p_2$ in $ \mathbb C^{n+2}$ 
of dimensions $1$ and $n$ respectively. In particular
the map $(p_1, p_2)\mapsto p_2$  from $\mathcal D_2^c$
to the Grassmannian manifold $Gr_n(n+2)$
realizes 
 $\mathcal D_2^c$ as the projectivization 
of the tautological bundle of 
$Gr_n(n+2)$.

\subsection{Cohomology groups
of Period domains}

Let $\mathcal X^c=SU(n+2)/S(U(n)\times U(2))$ 
be the compact dual of $\mathcal X$.
Then  $\mathcal X^c$  can be realized as Grassmannian manifolds
$Gr_2(n+2)$ of two planes in $\mathbb C^{n+2}$. 
Let $\pi: \mathcal D_1^c=SU(n+2)/S(U(n)\times U(1)\times U(1)),
 \mathcal D_2^c=SU(n+2)/S(U(n-1)\times U(1)\times U(2))\to 
\mathcal X^c$ be the natural fibrations.

\begin{Prop} \begin{enumerate}
\item The cohomology group $H^4(\mathcal D^c_1)$
 is three dimensional
and is generated by $\pi^*(\Omega^2), 
\pi^*(\omega), \hat{\Omega}^2
$.
\item Let $n\ge 3$. The cohomology group $H^4(\mathcal D^c_2)$
 is four  dimensional
and is generated by $\pi^*(\Omega^2), 
\pi^*(\omega),   \hat \Omega^2,
\hat \Omega\wedge \pi^*(\Omega).  $
\end{enumerate}
\end{Prop}
\begin{proof} The map $\pi$ defines
the twister space $\cal D_1^c$
as a $\bc\mathbb P^1=S^2$-bundle
over the Grassmannian manifold $\mathcal X^c$.
We recall the Gysin complex 
\cite[Proposition 14.33]{BT}
for the sphere covering
$\pi: \mathcal D_1^c\mapsto \mathcal X^c$,
$$
H^1(\mathcal X^c)
\stackrel{
\wedge e}{\to}
H^4(\mathcal X^c)
\stackrel{\pi^*}{\to}
H^4(\mathcal D_1^c) 
\stackrel{
\pi_*}
{\to}
H^2(\mathcal X^c) 
\stackrel{
\wedge e}
{\to}
H^5(\mathcal X^c)
$$
where $\wedge e$ is the multiplication by the Euler class $e$
of the sphere bundle, 
$\pi^*$ is the pull-back and 
$\pi_*$ is the integration along the fiber $S^2$.
Now $H^1(\mathcal X^c)=0, 
H^5(\mathcal X^c)=0$, 
 $H^2(\mathcal X^c)=\mathbb R$
and $H^4(\mathcal X^c)=\mathbb R^2$
since  the cohomology of
$\mathcal X^c$  is known; see e.g. \cite[Proposition 23.1]{BT}
for the computation of the cohomology in complex coefficients.
Thus the above sequence reduces to
$$
0\to
\mathbb R^2=H^4(\mathcal X^c)
\stackrel{
\pi^*}
{\to}
H^4(\mathcal D_1^c) 
\stackrel{
\pi_*}
{\to}
\mathbb R=H^2(\mathcal X^c) 
\to
0,
$$
from which we deduce that 
$H^4(\mathcal D^c) =\mathbb R^3$. It follows
further that $\pi^*$ is an injection. The square $\hat\Omega^2$
of the K\"ahler form is clearly not contained in 
${\pi^*}H^4(\mathcal X^c)$
 since its integration
along the fibers are nonzero, thus
$H^4(\mathcal D^c) $
is generated by $\pi^*(\omega), 
\pi^*(\Omega^2)$, and $\hat\Omega^2$. This proves (1).

Note that the map $\pi$ defines
the space $\mathcal D^c_2$
as the projectivization, $\mathbb P(L)\mapsto [L]$ of the tautological
bundle $L\to [L]$ of the Grassmannian $\mathcal X^c$
of $n$ dimensional subspaces $[L]$ in $\mathbb C^{n+2}$.
The $(1, 1)$-form $\hat \Omega$ restricted
to each fiber $\mathbb P(L)$ 
is  the Chern class $c_1(\mathbb P(v))$
of the  projective space. It follows from the Leray-Hirsch
theorem \cite[(5.11), (20.7)]{BT} or by \cite[(20.8)]{BT}
that $H^4(\mathcal D^c_2)$ is of dimension $4$ and
is generated by the four forms as claimed.
\end{proof}

\subsection{Pseudo-Riemannian metrics on Period domains}
Let $\cal X= SU(n,
2)/S(U(n)\times U(2))$ and $\cal D= SU(n,2)/K'$ where $K'$ is a subgroup of $K=S(U(n)\times U(2))$. The metric on $\cal X$ comes from
the Killing form on $\mathfrak g$ whose tangent space at $o=eK$ is identified with $\mathfrak p$ according to a Cartan decomposition
$\mathfrak g= \mathfrak t \oplus \mathfrak p$. Hence the metric on a period domain $\cal D$ comes from the Killing form
on $\mathfrak{ t}/\mathfrak{t'}\oplus \mathfrak p$ where $\mathfrak t'$ is the Lie algebra of $K'$. This metric is positive definite on the horizontal direction
$\mathfrak p$ which coincides with the metric on $\cal X$, negative definite on $\mathfrak{ t}/\mathfrak{t'}$ along the fibre direction of the projection $\pi:\cal D\ra\cal X$.  If $\Omega$ is a K\"a{}hler form on $\cal X$ defined by such a metric, $\hat\Omega$  pseudo-K\"a{}hler
form on $\cal D$, then on the horizontal direction of $T\cal D$, $\hat\Omega$ and $\pi^*\Omega$ coincide since the K\"a{}hler form is determined by the metric as in Equation (\ref{kahlerform}). We normalize a quaternionic
K\"a{}hler form $\omega$ on $\cal X$ so that its restriction to a copy of $H^2_\bc$ in $\cal X$ is equal to $\Omega^2$.

\section{Totally geodesic embeddings
of the complex hyperbolic space $B$ in
$\cal X$ and their possible holomorphic
liftings to period domains} 

We consider serval natural totally geodesic imbeddings
of the complex ball $B^m$ into the quaternionic
symmetric spaces and consider the corresponding
pull-back of the quaternionic $4$-forms
and the K\"ahler forms.
In \cite{GT} the authors
study  some holomorphic liftings of mappings
from the complex hyperbolic ball 
to quaternionic hyperbolic ball to holomorphic
mapping to the (pseudo-Hermitian) twister space, which enable
them to apply a variant of Schwarz lemma
and to prove  rigidity theorems. Following a
suggestion of Toledo we shall study holomorphic
liftings in our context.

\subsection{Holomorphic and totally real  imbeddings}\label{holo}
The complex hyperbolic space $H^n_\bc$, i.e.
the symmetric space $SU(n, 1)/U(n)$,
will be realized as the unit ball $B$ in $\mathbb C^n$
as in \S2.2.
A natural holomorphic embedding of $H^n_\bc=B=
\{(z_1,\cdots,z_n)\in \bc^n:\sum |z_i|^2 <1\}$ into $\cal X$ is given by
$$\rho:(z_1,\cdots,z_n)\hookrightarrow  Z=\begin{pmatrix}
                                    z_1 I_2\\
                                    z_2 I_2 \\
                                    \cdots\\
                                    z_n I_2\end{pmatrix}$$ which seems to give rise to the maximal Toledo invariant of $\omega$.
{ The push-forward on holomorphic tangent vectors at $0\in B$ is
 $$\rho_*: x=(x_1, \cdots, x_n)\in \mathbb C^n\mapsto
X=(x_1, x_1j,\cdots, x_n, x_nj)\in \bh^{2n}, $$  where $x_l=a_l+i b_l\in\bc$ and 
on which the form $\omega_j$ and $\omega_k$ vanish
and
$$\omega(X,Y,Z,W)= \omega_i\wedge \omega_i(X,Y,Z,W)$$$$=\text{Re}(iX\cdot \bar Y)\text{Re}(iZ\cdot \bar W)-   \text{Re}(iX\cdot \bar Z)\text{Re}(iY\cdot \bar W)+\text{Re}(iX\cdot \bar W)\text{Re}(iY\cdot \bar Z)$$ $$=\text{Re}(2i\sum_{i=1}^n x_i\bar y_i)\text{Re}(2i\sum_{i=1}^n z_i\bar w_i)- \text{Re}(2i\sum_{i=1}^n x_i\bar z_i)\text{Re}(2i\sum_{i=1}^n y_i\bar w_i)$$$$+\text{Re}(2i\sum_{i=1}^n x_i\bar w_i)\text{Re}(2i\sum_{i=1}^n y_i\bar z_i)    .$$But when we write $X=\begin{pmatrix}
                                                                0 & A\\
                                                                A^* & 0 \end{pmatrix}, Y=\begin{pmatrix}
                                                                0 & B\\
                                                                B^* & 0 \end{pmatrix}$
$$\Omega_o(X,Y)=g_o(JX,Y)=4 \text{Re Tr}(i B^*A)=4\text{Re} (2i\sum_{i=1}^n x_i\bar y_i).$$ Hence
      $$     \Omega_o^2(X,Y,Z,W)=16\omega(X,Y,Z,W)=4\Omega_B^2(x,y, z, w)$$ 
for tangent vectors $X=\rho_*(x),
Y=\rho_*(Y), Z=\rho_*(z), W=\rho_*(w)$ at the image of the natural holomorphic embedding of $H^n_\bc$. In other words,
      \begin{eqnarray}
\rho^*\Omega_o^2=16 \rho^*\omega=4\Omega_B^2, \quad
\rho^*\omega=\frac{1}{16}\rho^*\Omega_o^2=\frac 14\Omega_B^2,
      \end{eqnarray}
 for the natural holomorphic embedding $\rho$ of $H^n_\bc$ into $\cal X$.}

On the other hand, another  natural embedding
\begin{equation}
\label{eq:tot-real-pre} 
\lambda:
SU(n,1)\xhookrightarrow{} Sp(n,1)\xhookrightarrow{} SU(2n,2)
\end{equation}
 gives rise to a totally real embedding
\begin{equation}
  \label{eq:tot-real}
                          \lambda:
      (z_1,\cdots,z_n)\hookrightarrow  Z=\begin{pmatrix}
                                    \begin{pmatrix}
                                     z_1 & 0\\
                                     0 & \bar z_1 \end{pmatrix}\\
                                    \cdots\\
                                   \begin{pmatrix}
                                     z_n & 0\\
                                     0 & \bar
                                     z_n \end{pmatrix} \end{pmatrix}
\end{equation}
whose Toledo invariant of $\Omega_o$ is zero. Contrary to $SU(1,1)$ case, this totally real embedding is locally rigid for $n>1$, see \cite{KKP}.
On this totally real embedding, the tangent vectors are $X=(x_1, \bar x_1j,\cdots, x_n,\bar x_nj)\in \bh^{2n},\ x_i\in\bc$, and
$$\omega(X,Y,Z,W)= \omega_i\wedge \omega_i(X,Y,Z,W)=0.$$ Hence the Toledo invariant of $\omega$ also vanishes.  

{\bf Conjecture.}  These two special embeddings suggest that the Toledo invariant of $\omega$ is maximal on holomorphic embedding and zero on totally real embedding. More precisely, if a representation attains a maximum Toledo invariant, then it should be conjugate to the holomorphic embedding above, and if the quaternionic Toledo invariant is zero then it is conjugate to the totally real embedding. \\

{\bf Warning}: If we identify the holomorphic tangent space of $\cal X$ with $\bh^{2n}$ by $(x_1+ j y_1,\cdots,x_{2n}+j y_{2n})$,
$\omega$ vanishes on holomorphic embedding and $16\omega=\Omega_o^2$ on totally real embedding. Hence the convention determines which one has a maximal Toledo invariant. In \cite{GT}, it seems that they use a different convention from ours. Nevertheless we stick to our convention in this paper.

 \subsection{Symmetric square representation
of $SU(2, 1)$ and related
4-forms
}\label{symmetric}

Denote $V=\mathbb C^{2+1}
=\mathbb C^{2}
+\mathbb Ce_3
$ the space $\mathbb C^3$
equipped with the Hermitian metric with signature
$(2, 1)$ and  $B=SU(2, 1)/U(2)$ as in \S2.2.
Recall that it is also identified  as the open domain in $\mathbb P^2$
of lines $\mathbb C(z\oplus e_3)$
with negative metric, i.e. $|z|=|(z_1,z_2)|<1$.

Let $W=V^2$ be the symmetric square of $V$. Then
$W$ is equipped with the square of
the Hermitian metric of $V$ and $W=\mathbb C^4 +\mathbb C^2
=((\mathbb C^{2} )^2 +\mathbb C e_3^2)
\oplus (\mathbb C^{2} \odot e_3)
$
is of signature $(4, 2)$. Here $e_i\odot e_j=\frac{1}{2}(e_i\otimes e_j + e_j\otimes e_i)$.
We fix an orthonormal basis
$\{E_1, E_2, E_3, E_4, E_5, E_6\}$
of $W$ with
$$
E_j=e_j^2, E_4=\frac 1{\sqrt 2}(e_1\otimes e_2 +e_2\otimes e_1)=\sqrt 2 e_1\odot e_2,$$$$
 E_{4+i}=\frac 1{\sqrt 2}(e_3\otimes e_i +e_i\otimes e_3)=\sqrt 2 e_3\odot e_i,
j=1, 2, 3, i=1, 2.
$$

The square of the defining
representation of $H=SU(2, 1)
$ defines
a representation
$$
\iota: H\to G=SU(4, 2),\
g\mapsto \otimes^2 g
$$

As in \S2.2 the symmetric space $\mathcal X$ of
$SU(4, 2)$ will be realized as the open domain
of Grassmannian manifold $Gr(2, W)$ of
$2$-dimensional complex subspaces in $W$ with negative
metric, and is further identified
with the space of  $4\times 2$ matrices $Z$
with matrix norm $\Vert Z\Vert <1$ under the identification
$$
\{Zx\oplus x; x\in \mathbb C^2\} \mapsto Z.
$$
Recall also the normalization 
of the K\"ahler metric on $B$ and on $\mathcal X$
$$
g_B(u, v)=4\text{Re}(u_1 \bar v_1 +u_2 \bar v_2), \,
g_\mathcal X(u, v)=4\text{Re}\text {Tr} v^\ast u
$$
where the real tangent space of $B$ and 
$\mathcal X$ at $z=0$ and $Z=0$ are identified with $\mathbb C^2
$ and $M_{4\times 2}$; the respective K\"ahler forms
are  $\Omega_B(u, v)=g_B(iu, v)
$ and $\Omega_\mathcal X =g_\mathcal X(iu, v)$.

The representation $\iota: H\to G$
induces a totally geodesic mapping (with the same notation)
$\iota$: $B \to \mathcal X$. In terms of the above
identification of $B$ and $X$ as submanifolds
of projective and Grassmannian manifolds
the map $\iota$ is
$$
\iota (l)= l\odot l^{\perp}
$$
where $l^{\perp}$ is the orthogonal complement of
$l$ in $V$ and $l\odot l^{\perp}$
is the subspace of vectors
$u\otimes v + v\otimes u$, $u\in l, v\in l^{\perp}$.
We find now the map $\iota_\ast $ at $z=0\in B$.

Fixing the reference line $\mathbb Ce_3
\in \mathbb P^2$
and the plane $\mathbb C^2 \odot e_3
\in  Gr(2,W)$
corresponding to the point $0\in B$
and $0\in \mathcal X$,
the map $\iota$ is
$$
\iota: \exp(tX) \cdot (\mathbb Ce_3)
\mapsto (\exp(tX) \cdot (\mathbb Ce_3))
\odot (\exp(tX) \cdot (\mathbb C^2)),
$$  where
$$
X=\begin{pmatrix}
0 & 0 & a_1\\
0 & 0 & a_2 \\
\bar a_1 & \bar a_2 & 0 \end{pmatrix}
\in \mathfrak p
$$ and $\mathfrak{su}(2,1)=\mathfrak k + \mathfrak p$
is the Cartan decomposition.
Thus $\iota_\ast (X)
$ is the linear transformation
$$
\iota_\ast (X): \mathbb C^2 \to \mathbb C^4,$$$$
\mathbb C\{E_5, E_6
\}
\mapsto
\mathbb C\{(Xe_3)\odot e_1 +
e_3\odot (Xe_1),
(Xe_3)\odot e_2 +
e_3\odot (Xe_2)
\}.
$$

Note that $X e_1=\bar a_1 e_3,\ Xe_2=\bar a_2 e_3,\ Xe_3=a_1 e_1 + a_2 e_2$ and
$$(Xe_3)\odot e_1 +
e_3\odot (Xe_1)= (a_1 e_1 + a_2 e_2)\odot e_1 + e_3 \odot \bar a_1 e_3$$
$$=a_1 e_1\odot e_1 + a_2 e_2\odot e_1 +\bar a_1 e_3\odot e_3=a_1E_1+\bar a_1 E_3+\frac{a_2}{\sqrt 2} E_4.$$

A similar calculation for the second factor shows that, under the  basis
$\{E_j\}$,
$\iota_\ast(X)$ corresponds to the $4\times 2$ matrix
$$
\begin{bmatrix} a_1
&0\\
0& a_2 \\
 \bar a_1
& \bar a_2
 \\
 \frac{a_2}{\sqrt 2}
&  \frac{a_1}{\sqrt 2}
\end{bmatrix}
= T_{a},
$$
Taking the basis vectors $X=\begin{pmatrix}
                                  0 & 0 & 1\\
                                  0 & 0 & 0\\
                                  1 & 0 & 0
                                  \end{pmatrix}, Y=\begin{pmatrix}
                                  0 & 0 & i\\
                                  0 & 0 & 0\\
                                  -i & 0 & 0 \end{pmatrix},
Z=\begin{pmatrix}
                                  0 & 0 & 0\\
                                  0 & 0 & 1\\
                                  0 & 1 & 0 \end{pmatrix}, W=\begin{pmatrix}
                                  0 & 0 & 0\\
                                  0 & 0 & i\\
                                  0 & -i & 0 \end{pmatrix}$ we find the
corresponding images in $\bh^4$ under $\iota_\ast$
$$\iota_\ast(X)=(1,0,1,\frac{1}{\sqrt 2}j),\ \iota_\ast(Y)=(i,0,-i,\frac{k}{\sqrt 2}),$$$$\iota_\ast(Z)=(0,j,j,\frac{1}{\sqrt 2}),\ \iota_\ast(W)=(0,k,-k,\frac{i}{\sqrt 2})$$
and that
\begin{eqnarray}\label{value}
\omega(\iota_\ast(X), \iota_\ast(Y), \iota_\ast(Z), \iota_\ast(W))=\frac{11}{4}.
\end{eqnarray}

Namely
$$
\iota^\ast \omega=\frac{11}{64} \Omega_B^2
$$
where $\Omega_B$
is the K\"ahler form on $B$. { We can likewise compute 
$\iota^\ast \Omega^2$ and find
$$
\iota^\ast \Omega^2 =\frac 14 \Omega_B^2.
$$
}

Now there is a natural inclusion of $SU(4, 2)$ 
as a subgroup of $SU(n, 2)$, $n\ge 4$, and we will
also view $\iota$ as a homomorphism $\iota: SU(2, 1)\to
SU(4, 2)\to  SU(n, 2)$.

\subsection{Holomorphic lifting properties}

As  mentioned in the introduction it is of interests
to know if a hamonic map 
$f: B\to \mathcal X$ can be lifted to 
a holomorphic map into a period domain \cite{Carlson-et-al} (or Griffiths-Schmid domain) $\mathcal D$, 
 namely a homogeneous complex manifold $\mathcal D=G/L$ 
with a $G$ equivariant fibration $\pi: \mathcal D\to \mathcal X=G/K$.
We give an elementary criterion below.
		
\begin{Prop}
\label{crit}
Suppose there exists a holomorphic lifting $\hat f$ 
of a harmonic map $f: B\to \mathcal X$ 
to a period domain $\mathcal
D$.
Then $d^{(1, 0)}f(v) $ 
is  nilpotent for any  $v\in \mathfrak p^{+}=T_x^{(1, 0)}(B)$.
\end{Prop}
\begin{proof}
 Let $\hat f: B\to \mathcal D
$ be a holomorphic 
lift of  $f: B\to \mathcal X$. We can fix
a reference point $x=0$ and assume that $\hat f(0)=o=eL\in \cal D=G/L$.
The holomorphic tangent space of $T_o(\mathcal D)$ 
is given by the $\mathfrak n^+$-space  as in  Section \ref{period}
and is a nilpotent algebra of $\mathfrak g$. Now $f=\pi \circ \hat f$, 
and $f_*(0)(v)=\pi_*(o)(\hat f_*(0)(v)
)$, for $v\in T^{(1,0)}_0(B)$.  But
$\hat f(0)(v)\in \mathfrak n^+$,
since $\hat f$ is holomorphic, so 
$\hat f(0)(v)\in \mathfrak n^+$ is nilpotent
which implies $f_*(0)(v)$ is nilpotent since $\pi_*$
maps nilpotent elements to nilpotent elements where $\pi$
is the  quotient map $\mathcal D=G/L\to
\mathcal X=G/K$.
\end{proof}

We find  a holomorphic lift of the non-holomorphic map $\lambda$.

\begin{lemma}\label{real} 
The totally real imbedding  $\lambda: B\to \mathcal X$
can be lifted to a holomorphic horizontal mapping into the period 
domain $\mathcal D=SU(2n, 2)/S(U(2n)\times U(1)\times U(1)).$
    \end{lemma}
\begin{proof} Let $\mathbb C^{2n+2}
=\mathbb C^{n+1}\oplus
\mathbb C^{n+1}=\mathbb C^n_+ \oplus
\mathbb C^n_+ \oplus \mathbb C_-\oplus \mathbb C_-$
 be equipped with the Hermitian form $\langle, \rangle$
of signature $(2n, 2)$
with $\mathbb C^{n+1} $ being  of signature $(n, 1)$ as before, where
the sub-indices $\pm$ indicating the positivity or negativity  of the form.
We denote the standard basis as $\{e_1,\cdots,e_n, e_{n+1}\}$ for the first factor $\mathbb C^{n+1} $
and $\{f_{1},\cdots, f_{n+1}\}$ for the second summand $\mathbb C^{n+1} $.

Then according to the notation in Section \ref{negative},  the space $\mathcal X$ is the set of pairs of $2n$-coordinates $(z_1,\cdots,z_{n},w_1,\cdots,w_{n}), (z_1',\cdots,z_{n}',w_1',\cdots,w_{n}')$ such that
$$z_1 e_1 
+w_1 f_{1} 
+\cdots+ z_n  e_n 
+ w_n f_{n} 
+ e_{n+1},$$
$$ z_1' e_1 
+w_1' f_1 
+\cdots+z_n'  e_n 
+w_n'  f_{n} 
+ f_{n+1}
$$ represents a 2-dimensional negative definite subspace.

Consider the flag manfolds $\mathcal D$ of 
 pairs $(p_1, p_2)$, where $p_1$ is an one-dimensional
subspace  with negative form $\langle, \rangle$,
and $p_2$ is a $(2n+1)$-dimensional subspace with signature
$(2n, 1)$ containing $p_1$.
Then $\mathcal D$ is a $G=SU(2n, 2)$-homogeneous manifold and
$
\mathcal D= 
G/L, \quad L=S(U(2n)\times U(1)\times U(1))
.$ 
The homogeneity follows easily by elementary linear algebra. Fixing
the  point 
$p_0=(p_1, p_2)$, $p_1=\mathbb C e_{n+1}$
and $p_2=\mathbb C^{2n+1}=
\mathbb C^n_+  \oplus p_1  \oplus \mathbb C_+^n\oplus  0$
as a reference, then
the isotropic group of $p$ in $G$ is exactly 
$L$, proving the realization of
$\mathcal D$. The complex structure on $\mathcal D$
is realized as an open subset of
flag manifold $\mathcal D^c=SU(2n+2)/S(U(2n)\times U(1)\times U(1))=SL(2n+2, \mathbb C)/P$
 of all pairs $(p_1, p_2)$ of one dimensional
subspaces $p_1$ in $(2n+1)$-dimensional subspaces $p_2$, considered
as a homogeneous space of $SL(2n+2, \mathbb C)$ with
$P$ being a Borel subgroup as the isotropic subgroup
fixing the reference point $p_0=(p_1, p_2)=(\mathbb C e_{n+1},
\mathbb C^{2n+1})$ above.
The fibration $\pi: \mathcal D\to \mathcal X$ is then 
the map
$$
(p_1, p_2) \mapsto p_1\oplus p_2^{\perp};
$$
clearly  $p_1\oplus p_2^{\perp}$ is a two-dimensional
subspace in $\mathbb C^{2n+2}$ of signature $(0, 2)$, namely
it is an element in $\mathcal X$, and this map is $G$-equivariant.
Now we consider the map $\tilde \lambda: B\to \mathcal D$,
$$
\tilde \lambda(z)=(p_1, p_2);\ p_1=  (z_1e_1+\cdots + z_ne_n+e_{n+1}), 
p_2=  (\bar z_1 f_1+\cdots+\bar z_n f_n +f_{n+1})^{\perp},
$$
the orthogonal complement being
computed with respect the fixed indefinite form.
 It follows immediately from the formula that $\tilde\lambda$
is holomorphic in $z$. To be more precise,
 complex  coordinates near $p_0$ can be chosen
as
$$
(x, x', y, y')
\in \mathbb C^{n}\times \mathbb C^{n+1}\times \mathbb C^{n}
\times \mathbb C^{n}
\mapsto (p_1, p_2),
p_2=p_1\oplus q_2,
$$
$$
 p_1=
\mathbb C(e_{n+1}\oplus (x_1 e_1 +\cdots + x_n e_n
+ x_{1}' f_1  +\cdots + x_{n+1}' f_{n+1})),
$$
$$
 q_2=\text{span}\{ e_{1} +y_1 e_{n+1}, \cdots,
 e_{n} +y_n e_{n+1};\quad f_1 +
y_1' f_{n+1}, \cdots,
f_n+y_n' f_{n+1}\}.
$$
In terms of these coordinates the map $\tilde\lambda$
is 
$$\lambda: z=(z_1, \cdots, z_n)\mapsto 
(x, x', y, y') =(z, 0, 0, z)
$$
and is indeed holomorphic.
We have further
$$
\pi\circ \tilde\lambda:
z\mapsto (p_1, p_2)\mapsto p_1\oplus p_2^{\perp}
= \begin{pmatrix}
                                    \begin{pmatrix}
                                     z_1 & 0\\
                                     0 & \bar z_1 \end{pmatrix}\\
                                    \cdots\\
                                   \begin{pmatrix}
                                     z_n & 0\\
                                     0 & \bar
                                     z_n \end{pmatrix} \end{pmatrix}
$$
This corresponds precisely to the map $\lambda$
in   (\ref{eq:tot-real}).
\end{proof}

We consider now the  lifting property of $\iota$.

\begin{lemma}\label{non-lift} The above quadratic map $\iota: B\to \mathcal X$
does not lift to a holomorphic horizontal mapping into $\cal D_1=SU(4, 2)/S(U(4)\times U(1)\times U(1))$.
       \end{lemma}

\begin{proof}

Suppose $F$ is a holomorphic horizontal lifting.
The  complexification of $F_\ast$,
still denoted by $F_\ast$, maps $\mathfrak b^+$,
the holomorphic tangent space of $B$
to holomorphic tangent space $\mathfrak n^+$ (up
to changing of base point under SU(2)-action).
In particular the image 
of $\mathfrak b^+$  under $\iota_\ast$ is contained in
$\pi_\ast (\mathfrak n^+)$ where $\pi:\cal D_1 \ra \cal X$ is the natural projection. In particular
$\iota_\ast(\mathfrak b^+)
$ is a subspace of $\pi_\ast (\mathfrak n^+)$.
Using the above formula for $\mathfrak n^+$
we find that elements in $ \iota_\ast(\mathfrak b^+)\subset \pi_\ast (\mathfrak n^+)$\
are of the form
$$\begin{pmatrix}
0 & 0  & * \\
* &  0 & 0 \\
0 & 0 & 0 \end{pmatrix}.
$$

However our computations above show that
for
\begin{equation}
\label{horiz}
\begin{split}
S&=\begin{pmatrix}
0 & 0 & a_1\\
0 & 0 & a_2 \\
0 & 0 & 0 \end{pmatrix}\\
&=\frac 12\left(
\begin{pmatrix}
0 & 0 &  a_1\\
0 & 0 &  a_2 \\
 \bar a_1 &  \bar a_2 & 0 \end{pmatrix}
-\sqrt{-1}
\begin{pmatrix}
0 & 0 & i a_1\\
0 & 0 & i a_2 \\
-i \bar a_1 & -i \bar a_2 & 0 \end{pmatrix}
\right)
\in \mathfrak b^+,
\end{split}
\end{equation}
its image $\iota_\ast(S)$ is
$$
\iota_\ast(S)=
\begin{bmatrix}
0 & U\\
V& 0 \end{bmatrix}
$$
where
$$
U=
\begin{pmatrix}
a_1 & 0\\
0 & a_2 \\
0 & 0 \\
\frac{a_2}{\sqrt 2}
&\frac{a_1}{\sqrt 2}
 \end{pmatrix},\,
V=
\begin{pmatrix}
0 & 0 &a_1 &0\\
0 & 0& a_2 &0 \end{pmatrix}.
$$
This is a contradiction to the form
of $\pi_\ast(\mathfrak n^+)$.
\end{proof}

We may construct similarly the twister
cover $SU(2m,2)/S(U(2m)\times U(1)\times U(1))$
of $\mathcal X=SU(2m,2)/S(U(2m)\times U(2))$ as above and 
consider the question of holomorphic
lifting of maps from $B$ to $\mathcal X$. The
above proof leads to  a simple
necessary condition for the existence.
\begin{co}Given a representation $\rho:\Gamma\subset SU(n,1)\ra SU(2m,2)$, with a $\rho$-equivariant map $f$
on the associated symmetric spaces $B=SU(n,1)/S(U(n)\times U(1)),\ 
\mathcal X=SU(2m,2)/S(U(2m)\times U(2))$
 and a fixed base point $o=[K]\in SU(n,1)/S(U(n)\times U(1))$,
let
$$Df_o\begin{pmatrix}
  0 & X \\
  X^* & 0\end{pmatrix}= \begin{pmatrix}
                        0 & U \\
                        U^* & 0\end{pmatrix}$$ be a differential map at the base point, where $X\in \bc^n, U=(U_1,U_2)\in M_{2m\times 2}$.
                        For $f$ to have a holomorphic lift to the
                        twistor space, every component of  $U_1$  is
                        an conjugate $\mathbb C$-linear  in $X$, 
and every component of  $U_2$  is a
                        $\mathbb C$-linear in $X$.
Here we regard $Df_o$ as a map from $\bc^n$ to $M_{2m\times 2}=\bc^{4m}$.
\end{co}
\begin{proof}Note that $Df_o$ is a real linear map between real tangent spaces $T_o B$ and $T_{f(o)} \mathcal X$.
For $X=(z_1,\cdots, z_n)$ and $z_i=x_i+i y_i$, let $X=(x_1,\cdots, x_n, y_1,\cdots,y_n)=(x,y)$, with the same notation, be the corresponding coordinates in $\br^{2n}$. Then $i X$ corresponds to $$iX=(-y_1,\cdots,-y_n,x_1,\cdots,x_n)=(-y,x)$$ as usual.
For $f$ to lift to the holomorphic map to the twistor space, the equation (\ref{horiz}) should read
 $$Df_o(X-\sqrt{-1} iX)= (U_1',U_2')=(0, U_2').$$ Hence from $U_1'=0$, we get
 $$    \begin{pmatrix}
     A & B \\
     C & D \end{pmatrix} \begin{bmatrix}
                          x \\
                          y \end{bmatrix} -\sqrt{-1} \begin{pmatrix}
     A & B \\
     C & D \end{pmatrix} \begin{bmatrix}
                         -y \\
                          x \end{bmatrix}=0.$$ It is
$$         \begin{pmatrix}
           Ax+By \\
           Cx+Dy \end{pmatrix}+ \begin{pmatrix}
                                 -Cy+Dx \\
                                 -(-Ay+Bx)\end{pmatrix}=0.$$
From this we get
$$A=-D,\ B=C.$$  This exactly implies that every component function of
$U_1$ is 
conjugate $\mathbb C$-linear in $X=(z_1,\cdots,z_n)$ variables.
Using the equation for $U^*$, a similar calculation shows that every
component function
 of $U_2$ is $\mathbb C$-linear in $z_i$ variables for $f$ to have a holomorphic lift to the twistor space.
\end{proof}


We prove however that
the map $\iota$ can be
lifted to a holomorphic mapping to 
$\cal D_2=SU(4,2)/S( U(3)\times U(1)\times U(2))$ 

Let $f$ associate the triple
$(S^2 L^{\perp}, L^2, L\odot L^{\perp})$
 to a negative line $L$ in $V=\mathbb C^{2+1}$.
Then
 $S^2 L^{\perp}$ is a positive 3-dimensional space in $W$, $L^2$ is a
 positive line in $W$, and $L\odot L^{\perp}$ is a negative plane in
 $W$.
In the  explicit coordinates, if $L=\bc e_3$, then $L^\perp=\langle e_1, e_2\rangle$ and
$$S^2L^\perp
=\langle e_1^2, e_2^2, e_1\odot e_2\rangle=\langle E_1,E_2,
E_4\rangle,\ 
L^2=\langle e_3^2\rangle=\langle E_3\rangle,$$
$$L\odot L^\perp=\langle e_1\odot e_3, e_2\odot e_3\rangle=\langle
E_5,E_6\rangle.$$
 Hence the stabilizers 
of $ S^2 L^{\perp}, L^2, L\odot L^{\perp}$  are
$U(3), U(1)$ and $U(2)$ respectively. Therefore
$f: L \mapsto (S^2 L^{\perp}, L^2, L\odot L^{\perp})$ induces
a map
$$f: B\ra \cal D_2=SU(4,2)/S(U(3)\times U(1)\times U(2)).$$
Since $$\iota(L)=(L\odot L^{\perp},(L\odot L^{\perp})^\perp),$$ $f(L)=(  ( S^2 L^{\perp}, L^2),\iota(L))$ is a lifting of $\iota$ to $\cal D_2$.

We claim that $f$ is holomorphic with respect to a complex structure
on the period domain $\mathcal D_2$ introduced in Section \ref{period}.


Hence the claim follows from the fact that the holomorphic tangent vector in $B$
$$\label{hori}
S=\begin{pmatrix}
0 & 0 & a_1\\
0 & 0 & a_2 \\
0 & 0 & 0 \end{pmatrix}
=\frac 12\left(
\begin{pmatrix}
0 & 0 &  a_1\\
0 & 0 &  a_2 \\
 \bar a_1 &  \bar a_2 & 0 \end{pmatrix}\\
-\sqrt{-1}
\begin{pmatrix}
0 & 0 & i a_1\\
0 & 0 & i a_2 \\
-i \bar a_1 & -i \bar a_2 & 0 \end{pmatrix}
\right)
\in \mathfrak b^+
$$ is mapped to $\iota_\ast(S)$
$$
\iota_\ast(S)=
\begin{bmatrix}
0 & U\\
V& 0 \end{bmatrix}
$$ as in the proof of Lemma \ref{holo}
where
$$
U=
\begin{pmatrix}
a_1 & 0\\
0 & a_2 \\
0 & 0 \\
\frac{a_2}{\sqrt 2}
&\frac{a_1}{\sqrt 2}
 \end{pmatrix},\,
V=
\begin{pmatrix}
0 & 0 &a_1 &0\\
0 & 0& a_2 &0 \end{pmatrix}.
$$

Here we give another way to prove the liftability. Note that $\cal D_2$ can be identified with the open $SU(4,2)$ orbit in the homogeneous complex manifold $\hat{ \cal D}$
of partial flags consisting of lines inside 3-planes.  The stabilizer of the partial flag is $S(U(3)\times U(1)\times U(2))$. There is an obvious holomorphic map $F$ from $\bc \mathbb P^2$ to $\hat{\cal D}$, which associates the flag $l\odot l\subset l\odot \bc^{2,1}$ to a  line in $\bc^{2,1}$. The restriction of this map to $H^2_\bc\subset \bc\mathbb P^2$ is a holomorphic map. Furthermore the projection from $\cal D_2$ to $\cal X$ is 
$$l\odot l\subset l\odot \bc^{2,1} \ra  l\odot l^\perp$$ and hence $\iota=\pi\circ F$.

Now we show the horizontality, i.e., the image lies in the form $L\odot L^\perp$. For any smooth curve in $B$,
denote it by $L(t)=\langle v_0+ w(t) \rangle$ where $w(t)\subset v_0^\perp$, a differentiable family of lines, such that $w(0)=0, w'(0)\in v_0^\perp$. Then we can write $L(t)^\perp= \langle v(t)\rangle^\perp$ where $v(0)=v_0,\ v'(0)=w'(0)\in v_0^\perp$.

  Since $L(t)\odot L(t)^\perp$ is already horizontal,
it suffices to show the horizontality of $L(t)^2$ and $S^2(L(t)^\perp)$.
But $$L(t)^2=\langle (v_0 + w(t))\odot (v_0 + w(t))\rangle=\langle v_0^2+ v_0\odot w(t)+w(t)^2\rangle.$$ Hence
$$\frac{d}{dt}|_{t=0} L(t)^2=v_0\odot w'(0)\in L(0)\odot L(0)^\perp.$$
Similar calculation shows that
$$\frac{d}{dt}|_{t=0} S^2(L(t)^\perp)=\frac{d}{dt}|_{t=0} \langle
v(t)^\perp \odot v(t)^\perp \rangle$$
$$=\langle v'(0)^\perp \odot
v_0^\perp\rangle\subset \langle v_0\odot v_0^\perp\rangle\subset
L(0)\odot L(0)^\perp,$$
completing the proof.

\section{Character variety $\chi(\Gamma, SU(n,2))$}
\begin{thm}
There are at least 7 distinct connected components in 
$\chi(\Gamma, SU(n, 2))$, $n\ge 4$,  where $\Gamma\subset SU(2,1)$ is a uniform lattice in $SU(2, 1)$.
\end{thm}
\begin{proof} We view $SU(4, 2)$ as a subgroup of $SU(n, 2)$
as above. Let $X=\begin{pmatrix}
                                  0 & 0 & 1\\
                                  0 & 0 & 0\\
                                  1 & 0 & 0
                                  \end{pmatrix}, Y=\begin{pmatrix}
                                  0 & 0 & i\\
                                  0 & 0 & 0\\
                                  -i & 0 & 0 \end{pmatrix},
Z=\begin{pmatrix}
                                  0 & 0 & 0\\
                                  0 & 0 & 1\\
                                  0 & 1 & 0 \end{pmatrix}, W=\begin{pmatrix}
                                  0 & 0 & 0\\
                                  0 & 0 & i\\
                                  0 & -i & 0 \end{pmatrix}$ be the standard basis of $T_oB=\mathfrak p$ such that
      $$
\Omega^2_B(X,Y,Z,W)=\Omega_B(X,Y)\Omega_B(Z,W)=4\text{Tr}(YJX)\text{Tr}(WJZ)
=4\cdot 4=16.$$                           
Consider first the holomorphic
embedding  $\rho$ in Section \ref{holo}. 
The images of the above vectors under $\rho_*$,
written as block $3\times 3$-matrix with each entry being $2\times 2$
matrix, are
$$\rho_*(X)=\begin{pmatrix}
                                 0& 0& I_2\\
                                   0&  0 & 0 \\
                                  I_2& 0 & 0          \end{pmatrix},\ 
\rho_*(Y)=\begin{pmatrix}
                                 0 & 0 & i I_2\\
                                   0& 0 & 0 \\
                                 -iI_2 & 0 &0 
                                  \end{pmatrix},$$$$
\rho_*(Z)=\begin{pmatrix}
                                  0&0 &0  \\
                                  0&0  &  I_2 \\
                                  0 & I_2 & 0 \\
                                  \end{pmatrix},\ 
\rho_*(W)=\begin{pmatrix}
                                                    0& 0 & 0
\\
                                  0 & 0 & iI_2 \\
                                  0 &  -iI_2 & 0 \\
 \end{pmatrix},$$ which correspond to
$$\rho_*(X)=(1,j,0,0),\ \rho_*(Y)=(i,ij,0,0),$$$$\rho_*(Z)=(0,0,1,j),\ \rho_*(W)=(0,0,i,ij)$$ in $\bh^2$ coordinates, see Section \ref{Quaternionic}.
Then by Equation (\ref{kahler}) $$\rho^*\omega(X,Y,Z,W)=4, \ \text{i.e.}\ \rho^* \omega=\frac{1}{4} \Omega_B^2$$ whereas for the square representation $\iota$, by Equation (\ref{value})
$$\iota^*\omega(X,Y,Z,W)=\frac{11}{4},\ \text{i.e.}\ \iota^*
\omega=\frac{11}{64} \Omega_B^2.$$ For the totally real embedding
  (\ref{eq:tot-real}), the pull-back form vanishes. This implies that
the quaternionic Toledo invaraints are
$$\int_{\Gamma\backslash H^2_\bc} \rho^*\omega=\frac{1}{4}\int_{\Gamma\backslash H^2_\bc} \Omega_B^2=\frac{1}{4}\vol(\Gamma\backslash H^2_\bc),$$$$\ \int_{\Gamma\backslash H^2_\bc} \iota^*\omega=\frac{11}{64}\int_{\Gamma\backslash H^2_\bc} \Omega_B^2=\frac{11}{64}\vol(\Gamma\backslash H^2_\bc),\ 0$$ respectively.

The last representation with a different Toledo invariant is given by the embedding $\phi:(z_1,\cdots,z_n)\ra ((z_1,0),\cdots,(z_n,0))$ which produces that
\begin{equation}
  \label{eq:phi-p-b}
\phi^* \omega=\frac{1}{16} \Omega_B^2.  
\end{equation}
Since the quaternionic Toledo invariant is constant on each connected component, we get 4 different connect components. By taking
the complex conjugate of $\rho, \iota $ and $\phi$
we get then 7 components. This completes the proof. 
\end{proof}
Note that for a lattice $\Gamma\subset SU(2,1)$, the holomorphic
embedding 
$\rho$ corresponds to the diagonal embedding
$\gamma\ra (\gamma,\gamma)\in SU(2,1)\times SU(2,1)\subset SU(4,2)$, and the totally real embedding to
$\gamma\ra(\gamma,\overline\gamma)$ whereas the last example in
the previous theorem corresponds to the embedding $\gamma\ra (\gamma,id)\in SU(2,1)\times SU(2,1)\subset SU(4,2)$.

In this direction, Toledo constructed the following examples \cite{Domingo}.
There are examples of two complex hyperbolic surfaces $X=\Gamma\backslash H^2_\bc$ and $Y=\Gamma'\backslash H^2_\bc$ with a surjective holomorphic map
$f:X\ra Y$ with $0<\deg(f)< \frac{\vol(X)}{\vol(Y)}$, which  induces a group homomorphism $f_*: \Gamma\to \Gamma'$.
See also \cite{DM, Mos} for the constructions of various
subgroups $\Gamma'\subset \Gamma$ of finite index.
(The volumes $\vol(X)$ and $\vol(Y)$ can  be further computed
by using the Chern-Gauss-Bonnet theorem for orbifolds.)
Consider the following representation
$$
 \Gamma\stackrel{f_*}{\rightarrow} \Gamma'\stackrel{\phi}{\rightarrow}
SU(4,2),
$$
 where $\phi$ is the restriction of the holomorphic embedding (5)
above.
Then the quaternionic Toledo invariant of this representation is 
$$\int_X f^*(\phi^*\omega)
=\int_X
f^*(\frac{1}{16}\Omega^2_B)
=\frac{1}{16}
{\deg(f)\vol(Y)} < \frac{1}{16}
{\vol(X)}, 
$$
with $\frac{1}{16}
{\vol(X)}$ being the smallest  among the Toledo invariants
in Theorem 4.1 except zero case.
We obtain thus
an improvement of Theorem 4.1 in this case, viz

\begin{Prop}
Let $\Gamma\subset \Gamma'$ be as above. There exist at least
$9$  distinct components in $\chi(\Gamma,SU(4,2))$.
\end{Prop}

Some  versions of local rigidity for the representations
in some of the components above have been studied in 
\cite{KKP, Klingler-inv}.


\section{Milnor-Wood inequality and Global rigidity for quaternionic Toledo invariant}
In this section we show that if there exists a holomorphic horizontal lifting, then the Milnor-Wood type inequality holds with a quaternionic K\"ahler form.
In this section, we normalize the metrics on $H^2_\bc$ and on $\cal X=SU(2n,2)/S(U(2n)\times U(2))$ so that the holomorphic sectional curvatures are equal to $-1$.
\begin{lemma}Let $\cal D$ be a period domain of $\cal X$ with a pseudo-K\"ahler metric such that it is negative definite on vertical directions and positive definite on horizontal directions. Its associated pseudo-K\"ahler form is $\hat\Omega$ which agrees
with $\pi^*(\Omega)$ on the horizontal direction where
the K\"ahler form on $\cal X$ is denoted $\Omega$. If $f:H^2_\bc\ra \cal D$ is a   horizontal holomorphic map, then Schwarz lemma holds, i.e., $f^*(\hat\Omega)\leq \Omega_B$ where $\Omega_B$ is the K\"ahler form on $H^2_\bc$.
Equality holds at every point if and only if $f$ is a
horizontal holomorphic geodesic embedding of $H^2_\bc$ in $\cal D$.
\end{lemma}
\begin{proof}The proof is exactly the same as the one given in 
Theorem 3.3 in \cite{GT}. The idea is as follows. First consider the case a mapping from the hyperbolic plane 
$H^1_\bc$, $f:H^1_\bc\ra \cal D$. If $f^*\hat\Omega=u\Omega_{B^1}$. Then by the method of Section 2 of Chapter I, III of \cite{Koba},  one can show that $u\leq 1$. If equality holds at every point, then $f$ is an isometric immersion. If $M$ is the image and $\alpha$ is the second fundamental form, then since both holomorphic sectional curvatures are $-1$, one can show that $\alpha=0$, consequently $f$ is a totally geodesic holomorphic embedding. For $f:H^2_\bc\ra\cal D$ case, by considering all hyperbolic hyperbolic planes
$H^1_\bc$ in $H^2_\bc$, one concludes that the second fundamental form vanishes, hence totally geodesic embedding.
\end{proof}
\begin{Prop}Let $M=\Gamma\backslash H^2_\bc$. Suppose $\rho:\Gamma\ra SU(n,2)$ is a representation whose associated $\rho$-equivariant harmonic map
$f:B\ra \cal X$ lifts to a holomorphic horizontal map $\hat f$ to $\cal D$. 
Then
the Milnor-Wood type inequality holds. {If equality holds, then it is a holomorphic embedding.}
\end{Prop}
\begin{proof}Since $H^4(M,\br)=\br$, the pull-back of 4-forms to $M$ are all proportional to each other up to exact forms.
Specially $$f^*\omega=\hat f^*(\pi^*\omega)=c \hat f^*( \pi^*\Omega^2)+d\alpha=c f^*\Omega^2+d\alpha.$$
A K\"ahler form $\hat\Omega$ of $\cal D$ agrees with $\pi^*(\Omega)$ on horizontal directions, hence $\hat f^*(\hat \Omega)=\hat f^* (\pi^* \Omega)=f^*\Omega$.
But since $\hat f$ is holomorphic, by Schwarz Lemma,
$$f^*\Omega=\hat f^*( \pi^*\Omega) \leq \Omega_B.$$ Hence
$$\frac{1}{c}\int_M f^*\omega=\int_M f^*\Omega^2\leq \int_M \Omega^2_B=vol(M).$$

Now since we normalize $\omega$ so that its restriction to complex 2-dimensional hyperbolic space is equal to $\Omega^2$, we have $c=1$ and
$\hat f^*(\pi^*\omega)= \hat f^*( \pi^*\Omega^2)+ d\alpha$, and consequently the Milnor-Wood inequality
$$\int_M f^*\omega=\int_M f^*\Omega^2\leq vol(M).$$
Suppose $\int_M f^*\omega=vol(M)$. Then $f^*\Omega^2=\Omega^2_B$ pointwise, which implies that $f$ is a holomorphic embedding by the previous lemma.
\end{proof}

%


\begin{thebibliography}{99}
{\small}


\bibitem{BT} R. Bott and L. Tu, Differential forms in algebraic 
topology, Springer-Verlag, 1982.



\bibitem{BI1} M. Burger and A. Iozzi, Bounded cohomology and representation varieties of lattices in $PU(1,n)$, preprint announcement, 2000.

\bibitem{BI2} M. Burger and A. Iozzi, A measurable Cartan theorem and applications to deformation rigidity in complex hyperbolic geometry, {\it Pure Appl. Math. Q.}, 4(1, Special Issue: In honor of Grigory Margulis. Part 2): 181-2-2, 2008.



\bibitem{Carlson-et-al}
J. Carlson, S. M\"u{}ller-Stach, and C. Peters,
   {Period mappings and period domains},
{Cambridge Studies in Advanced Mathematics},
{Vol. 85},
{Cambridge University Press, Cambridge}, {2003}.



\bibitem{DM}P. Deligne and G. Mostow, 
Monodromy of hypergeometric funcgions and non-lattice
integral monodromy groups, 
{\it
 Publ. Math. IHES,}
   {\bf 63
} (1986), 
5-90.



\bibitem{GT}O. Garc\'ia-Prada and D. Toledo,  A Milnor-Wood inequality
  for complex hyperbolic lattices in quaternionic space, 
{\it 
Geom \& Topology,}
  {\bf
 15}
(2011), no. 2, 1013-1027.

    \bibitem{GKL}W. Goldman, M. Kapovich and B. Leeb, Complex
      hyperbolic manifolds homotopy equivalent to a Riemann surface, 
{\it Comm. Anal. Geom.,}
  {\bf 9}
 (2001), 61-95.

\bibitem{Her91}L. Hern\'andez, Maximal representations of surface groups in bounded symmetric domains, {\it Trans. Amer. Math. Soc.}, {\bf 324} (1) (1991), 405-420.

\bibitem{KKP}I. Kim, B. Klingler and P. Pansu, Local quaternionic rigidity for complex hyperbolic lattices,
{\it 
 Journal of the Institute of Mathematics of Jussieu.
}
  {\bf 
11}
 (2012), no 1, 133-159.


\bibitem{Klingler-inv} {B. Klingler},
   {Local rigidity for complex hyperbolic lattices and Hodge theory},
{\it  Invent. Math.},
   {\bf 184} ({2011}),
   no.{3},   {455--498.}
\bibitem{Koba} S. Kobayashi, Hyperbolic manifolds and holomorphic mappings, Pure and Applied
Math. 2, Marcel Dekker, New York (1970).
\bibitem{KM} V. Koziarz and J. Maubon, Representations of complex hyperbolic lattices into rank 2 classical Lie
groups of Hermitian type, {\it Geom. Dedicata.}, {\bf 137} (2008), 85-111.

\bibitem{Mos}G. Mostow,
Monodromy of hypergeometric functions and nonlattice integral
monodromy.
{\it 
Publ. Math. IHES,
}
 {\bf
  63}
 (1986), 
5�V89.

\bibitem{Po}M. B. Pozzetti, Maximal representations of complex hyperbolic lattices into $SU(m,n)$, GAFA. 25 (2015), 1290-1332.
\bibitem{Sal}S. Salamon, Quaternionic K\"ahler manifolds,
 {\it 
 Invent. Math.,
}
  {\bf 67} (1) (1982), 


143-171.
\bibitem{Sa}I. Satake, Algebraic structures of symmetric domains, Kano Memorial Lectures 4, Iwanami Shoten, Tokyo, Princeton University Press, Princeton. NJ, 1980.
\bibitem{Tol89}D. Toledo, Representations of surface groups in complex hyperbolic space, {\it J. Differential Geom.}, {\bf 29} (1) (1989), 125-133.
\bibitem{Domingo}D. Toledo, Maps between complex hyperbolic surfaces,
  Special volume dedicated to the memory of Hanna Miriam Sandler. 
 {\it 
Geom. Dedicata,}
  {\bf 97} (2003), 115-128.

\bibitem{Xia}E. Xia, The moduli of flat $PU(2,1)$ structures on
  Riemann surfaces, 
 {\it 
Pacific J. Maths.},
 {\bf 195} (2000) 231-256.

    \end{thebibliography}
     \end{document}